\documentclass{amsart}
\usepackage{latexsym,amsmath,amsthm,amssymb}

\usepackage{color}

\usepackage{graphicx}

\newtheorem{theorem}{Theorem}[section]
\newtheorem{lemma}[theorem]{Lemma}
\newtheorem{proposition}[theorem]{Proposition}

\theoremstyle{remark}
\newtheorem{example}[theorem]{Example}
\newtheorem{remark}[theorem]{Remark}

\theoremstyle{definition}
\newtheorem{definition}[theorem]{Definition}

\newtheorem{problem}[theorem]{Problem}

\def\Nat{\mathbb{N}}

\newcommand{\dR}{\ensuremath{\mathbb{R}}} 
\newcommand{\R}{\dR}

\newcommand{\PP}{{\mathcal P}}

\newcommand{\Exp}{{\rm I\!E}}

\begin{document}

\title{Optimal transportation of processes with infinite Kantorovich distance. Independence and symmetry.}

\author{Alexander~V.~Kolesnikov}
\address{ Higher School of Economics, Moscow,  Russia}
\email{Sascha77@mail.ru}

\vspace{5mm}
\author{Danila~A.~Zaev}
\address{ Higher School of Economics, Moscow,  Russia}
\email{zaev.da@gmail.com}

\thanks{ 
The first named author was supported by RFBR project 14-01-00237 and  the DFG project  CRC 701.
This study (research grant No 14-01-0056) was supported by The National Research University-Higher School of Economics' Academic Fund Program in 2014/2015.
 The second author was partially supported by AG Laboratory NRU-HSE, RF government grant, ag.  11.G34.31.0023.
}


\begin{abstract}
We consider probability measures on $\mathbb{R}^{\infty}$ and study
 optimal transportation mappings for the case of infinite
Kantorovich  distance. Our examples include 1) quasi-product measures,
2) measures with certain
symmetric properties, in particular,  exchangeable and stationary measures.
We show in the latter case  that existence problem for optimal transportation
is closely related to ergodicity of the target measure.
In particular, we prove existence of the symmetric optimal transportation for a certain class
of stationary Gibbs measures.
\end{abstract}

\maketitle

\section{Introduction}

Let us consider two Borel probability measures $\mu, \nu$ on $\mathbb{R}^d$.
The central result (Brenier theorem) of the finite-dimensional optimal transportation theory
establishes under fairy general assumptions existence of the corresponding optimal transportation mapping $T$, which can be characterized by 
the following properties:

1) $T=\nabla \varphi$, where $\varphi$ is a convex function

2) $\nu$ is the image of $\mu$ under $T$: $\nu = \mu \circ T^{-1}$.

The mapping $T$ exists, in particular, when both measures are absolutely continuous and have finite second moments.
The second assumption can be replaced by the weaker assumption of the finiteness of the corresponding Kantorovich distance $W_2(\mu,\nu)$
but it is does not make much difference for the finite-dimensional problems.
However, this difference  becomes  essential in the infinite-dimensional case.

It is well-known that the optimal transportation mapping $T$
  solves the so-called Monge problem, meaning that $T$ gives minimum
  to the functional
  $$
    \int_{\mathbb{R}^d} \| r(x) -x \|^2 d \mu(x)
  $$
among of the mappings $r \colon \mathbb{R}^d \mapsto \mathbb{R}^d$ pushing forward $\mu$ onto $\nu$; here
$\| \cdot\|$ is the standard Euclidean norm.  The corresponding minimal value 
coincides with the squared Kantorovich
distance $W^2_2(\mu,\nu).$

Now let us consider a couple of measures on an infinite-dimensional linear space $X$; to avoid unessential technicalities,
we will assume everywhere that $X= \mathbb{R}^{\infty}$.
We deal throughout with the standard
Hilbert norm
$$
\|x \|^2 : = \|x\|^2_{l^2} = \sum_{i=1}^{\infty} x^2_i,
$$
which takes infinite value almost everywhere with respect to most of the measures we are interested in. 

  What is a natural analog of the Brenier theorem in this setting?
To understand the situation better let us consider  the Gaussian model.

\begin{example}
Let $\gamma = \prod_{i=1}^{\infty} \gamma_i =  \prod_{i=1}^{\infty} \frac{1}{\sqrt{2 \pi}} e^{-\frac{x^2_i}{2}} \ d x_i$ be the standard Gaussian product measure on $\R^{\infty}$ and $H=l^2$
be the corresponding Cameron--Martin space. More generally, one can  consider any
abstract Wiener space.

The optimal transportation problem is well-understood for the case of measures $\mu$ and $\nu$ which are absolutely continuous with respect to  $\gamma$. The most general results were obtained in \cite{FU1}
(another approach has been developed in \cite{Kol04}).
In particular, for a broad class of  probability measures $f \cdot \gamma$ absolutely continuous w.r.t. $\gamma$ there exists a transportation mapping $T(x) = x + \nabla \varphi(x)$
minimizing the cost
$$
\int \| T(x) - x \|^2_{l^2}  \ d \gamma
$$
and pushing forward $\gamma$ onto $f \cdot \gamma$. Analogously, there exists a transportation mapping pushing forward $f \cdot \gamma$ onto $\gamma$. The gradient operator $\nabla$ is understood with respect to $\langle \cdot, \cdot \rangle_{l^2}$-scalar product.

It is known (this follows from the so-called Talagrand transportation inequality) that under assumption $\int f \log f \ d \gamma < \infty$  the Kantorovich distance between $\gamma$ and $f \cdot \gamma$ is finite
$$W^2_2(\gamma, f \cdot \gamma) = \int \| T(x) - x \|^2_{l^2}  \ d \gamma < \infty.
$$
In particular, $\nabla \varphi(x) \in l^2$ for $\gamma$-almost all $x$.
More on optimal transportation on the Wiener space, the corresponding Monge--Amp{\'e}re equation,
regularity issues, and transportation on other infinite-dimensional spaces
see  in  \cite{BoKo2005}, \cite{BoKo2011}, \cite{Caval},
 \cite{FangShao}, and \cite{FangNolot}.
\end{example}

It this paper we study situation when the Kantorovich distance between measures is a priori {\bf infinite}.  
This makes   impossible in general to understand $T$ as a solution
to a certain minimization problem.
Nevertheless, we have many good candidates to be called ''optimal transportation''
in many particular cases. The following example motivates our study.

\begin{example}
\label{simple-ex}
1) Let  $\mu = \prod_{i=1}^{\infty} \mu_i(dx_i)$, $\nu = \prod_{i=1}^{\infty} \nu_i(dx_i)$ be product probability measures.
Assume that all $\mu_i$ have densities. Then  there exists a mass transportation mapping $T$ pushing forward $\mu$ onto $\nu$ which has the form
$$T(x) = (T_1(x_1), \cdots, T_i(x_i), \cdots),$$
where $T_i(x_i)$ is the one-dimensional optimal transportation pushing forward $\mu_i$ onto $\nu_i$.

2)
Let us consider the Gaussian  measure
$\mu$ which is a push-forward image of the  standard Gaussian measure $\gamma$ under a linear mapping $T(x) =Ax$ with $A$ symmetric and positive.
It is well-known (and can be obtained from the  law of large numbers) that $\gamma$ and $\mu$ are mutually singular even in the simplest case $A=2 \cdot \mbox{Id}$.
$T$ is "optimal" because it is linear and  given by a positive symmetric operator. Heuristically,
$$
T(x) = \frac{1}{2} \nabla  \langle Ax, x \rangle.
$$
 It is clear that in both cases $T$ cannot be obtained as a minimizer of a
 functional of the type $\int \|T(x)-x\|^2_{l^2} \ d \mu$.
\end{example}

We state now  the central problem of this paper.

\begin{problem}
\label{mainproblem}
Let $\mu$ and $\nu$   be two probability measures  on $\mathbb{R}^{\infty}$.
 When does  exist a  transportation mapping $T$
 pushing forward  $\mu$ onto $\nu$ which is
''optimal'' for the cost function $c(x,y)=\|x-y\|^2_{l^2}$?
\end{problem}

In this paper we deal with two model situations.

{\bf Quasi-product measures.}

We assume that both measures have densities with respect to product probability measures
$$
\mu = f \cdot \mu_0, \ \nu = g \cdot \nu_0,
$$
$$
\mu_0 = \prod_{i=1}^{\infty} \mu_i(dx_i), \nu_0 = \prod_{i=1}^{\infty} \nu_i(dx_i).
$$
Then the corresponding  "optimal tranportation" is a small perturbation
of the diagonal mapping, considered in Example \ref{simple-ex}.

{\bf  Symmetric measures.}

It is possible to give a meaning to the Monge--Kantorovich optimization problem if we restrict ourselves to  a certain class of symmetric measures. In this paper we
consider two types of symmetry: exchangeable measures (invariant with respect to finite permutations of coordinates) and stationary measures on $\mathbb{R}^{\infty}$ (invariant with respect to shifts of coordinates).
Note that  $\|x-y\|^2_{l^2}$ is symmetric with respect to both types of symmetry. More generally, let $G$ be a group of linear operators  which acts
on $X=Y = \mathbb{R}^{\infty}$ and $X\times Y$: $x \to g x$, $(x,y) \to (gx, gy)$, $g \in G$ and preserves  the cost function $c(x,y)$.
We assume that every basic vector $e_j$ can be obtained from any other $e_i$ by 
action of this group: there exists $g \in G$ such that $e_i = g e_j$.
Note that under these assumptions all the coordinates are identically distributed.
This leads us to the following definition:
given $G$-invariant marginals $\mu$ and $\nu$ we call $\pi$ an
optimal (symmetric, invariant) solution to the Monge--Kantorovich problem if
$\pi$ solves the Monge--Kantorovich problem
$$
\int (x_1-y_1)^2 \ d \pi \to \min
$$
among all of the measures which are invariant with respect to $G$.
If  there exists a mapping $T$ such that its graph $\Gamma = \{x, T(x)\}$
satisfies $m(\Gamma)=1$, we say that $T$ is an
optimal transportation mapping pushing forward $\mu$ onto $\nu$.

The following counter-example, however, demonstrates that the 
optimal transportation may fail to exist by a quite simple reason.

\begin{example}
Let $\mu = \gamma$ be the standard Gaussian measure on $\mathbb{R}^{\infty}$ and
$$\nu = \frac{1}{2} (\gamma + \gamma_2)$$ be the average of $\gamma$ and its homothetic image
$\gamma_2 = \gamma \circ S^{-1}$, where $S(x)=2x$. There is no any mass transportation  $T$ of $\mu$ to $\nu$ which  commutes with any cylindrical rotation.
Indeed, any mapping of such a type must have the form $T(x) = g (x) (x_1, x_2, \cdots) = g(x) \cdot x$, where $g$ is invariant with respect to any "rotation", in particular, with respect to any coordinate permutation.
But any function $g$ of this type is constant $\gamma$-a.e. This is a corollary of the Hewitt--Savage $0-1$ law. It is clear that there is no any mass transportation of this type for the given target measure.
\end{example}

There is a general principle behind of this simple example. Recall that a measure $\mu$ is called ergodic with respect to a group action $G$, if for every $G$-invariant set $A$
one has either $\mu(A)=1$ or $\mu(A)=0$. It follows directly from the definition that {\it there does not exists a  bijective mass transportation $T$ pushing forwarg $\mu$ onto $\nu$, such that $T \circ g = g \circ T$ for every $g \in G$,  provided $\mu$
 is $G$-ergodic but $\nu$ is not}.

This observation leads to the following problem.

{\bf Problem.}
Let $G$ be a group of linear operators acting on $\mathbb{R}^{\infty}$ and preserving $l_2$-distance
(model example: group of shifts). Let $\mu, \nu$ be {\bf ergodic} $G$-invariant measures.
When does exist a transportation $T \colon \mathbb{R}^{\infty} \mapsto \mathbb{R}^{\infty}$
pushing forward $\mu$ onto $\mu$, which commutes with $G$ and gives minimum to the Monge functional
$ T \mapsto \int_{\mathbb{R}^{\infty}} (T_1(x) - x_1)^2 \ d\mu$?

Trivially, the ergodicity by itself is not sufficient for the affirmative answer to this problem. In addition to it, we need to have certain infinite-dimesional 
analogs of  ''absolute continuity'' for the source measure $\mu$.

We believe that the symmetric transportation problem must have deep
and very interesting relation with the ergodic theory. The second named author studied the interplay between ergodic decompositions and transportation theory  in \cite{Zaev2}.
Another  interesting connection has been established in  \cite{B}. It was shown that the Birkhoff ergodic theorem implies
  equivalence between optimality and the so-called  cyclical monotonicity property. 
The related problems on optimal transportation in symmetric settings have been considered in \cite{RS} (stationary processes), in \cite{Vershik} (symmetric measures on graphs),
and in \cite{LM}, \cite{LOT}, \cite{CLO} (ergodic theory). 
Transportation problems with symmetries  have been studied in \cite{GhMau}, \cite{Moameni}. Further development of the  duality theory for transportation problem  with linear restriction has been obtained  in   \cite{Zaev1}.

The paper is organized as follows: in Section 2 we give preliminaries in transportation theory, ergodic theory, and recall some important results 
on log-concave measures. In Section 3 we establish sufficient conditions for existence
of optimal transportation mappings which are obtained as a.e.-limits of finite-dimensional  approximations. The applications of this result are obtained in Section 4.
Here we prove existence of optimal transportation for a couple of measures having densities with respect to product measures.
In Section 5 we discuss the invariant optimal transportation problem, consider examples and prove some basic facts.
In Section 6 we briefly discuss Kantorovich duality for problem which is invariant with respect to the action of a group.
In Section 7 we construct a non-trivial example of a symmetric optimal  transportation $T$. Namely, we establish  sufficient conditions for existence of $T$
pushing forward a stationary measure into the standard Gaussian measure. Finally, we apply this result to a certain class of Gibbs measures.

\newpage 

\section{Preliminaries}

\subsection{Optimal transportation problem}

\

{\bf Kantorovich problem.} Given two probability measures $\mu$ and $\nu$ on the spaces $X$ and $Y$ respectively,
and a cost function $c: X \times Y \mapsto \mathbb{R} \cup \{+\infty\}$ we are looking for the  minimum of the functional
$$
W^2_2(\mu,\nu) = \inf \Bigl\{ \int \|x-y\|^2 \ d m \colon m \in P(\mu,\nu)  \Bigr\},
$$
on the space $P(\mu,\nu)$ of  probability measures  with fixed projections:
$Pr_X m=\mu, Pr_Y m = \nu$.

In the classical setup $X=Y = \mathbb{R}^n$, $c=|x-y|^2$
the solution  $m$ is supported on the graph of a mapping $T : \mathbb{R}^n \mapsto \mathbb{R}^n$:
$$
m(\Gamma)=1, \ \ \mbox{where} \ \ \Gamma = \{(x,T(x)), \ x \in \R^d\}.
$$
(see \cite{AGS}, \cite{BoKo2012}, \cite{Vill}.).
The functional $W_2(\mu,\nu)$ is a distance in the space of probability measures.
In what follows we call it the Kantorovich distance.
The mapping $T$ is called optimal transportation of $\mu$ onto $\nu$.

Another well-known fact which  will be used throughout the paper is the following
relation called the Kantorovich duality:
$$
W_2(\mu,\nu)=-\frac{1}{2}J(\varphi,\psi),
$$
where $$J(\varphi,\psi) = \inf_{\varphi,\psi} \Bigl\{\int \Bigl( \varphi(x)  - \frac{x^2}{2} \Bigr)\ d \mu + \int  \Bigl(\psi(y) - \frac{y^2}{2} \Bigr) \ d \nu, \ \  \varphi(x) + \psi(y) \ge \langle x, y \rangle\Bigr\},$$
where the infimum is taken over couples of integrable Borel functions $\varphi(x), \psi(y)$.
The function $\varphi$ in the dual problem coincides with the potential generating the transportation mapping
$$T = \nabla \varphi.$$

\subsection{Ergodic decomposition}

Given a Borel transformation $S : X \mapsto X$ of the space $X$ we call a Borel probability measure $\mu$ ergodic 
if any $S$-invariant measurable set $A$ has the property $\mu(A)=1$ or $\mu(A)=0$.
A similar tegrminology is used if instead of a single mapping $S$ we deal with a family $G$ of transformations.

The ergodic $G$-invariant measures are extreme points of the set of all $G$-invariant measures, hence
any $G$-invariant measure can be represented as the average of $G$-invariant ergodic measures.
The famous de Finetti theorem establishes decomposition of this type for a class of exchangeable measures, i.e.
measures, invariant with respect to a permutation of a finite number of coordinates.

\begin{theorem}
Let $\mathcal{P}$ be the space of  Borel probability measures on $\mathbb{R}$
equipped with the  weak topology. Then for every Borel exchangeable $\mu$ on $\mathbb{R}^{\infty}$
there exists a Borel probability measure $\Pi$ on $\mathcal{P}$ such that
$$
\mu(B) = \int m^{\infty}(B) \Pi(dm),
$$
for every Borel $B \subset \mathbb{R}^{\infty}$.
\end{theorem}

Yet another example of the ergodic decomposition where a precise description is possible is given by
rotationally invariant measures (see Example \ref{Rim}). 

\subsection{Log-concave measures and functional inequalities}

We recall that a probability measure $\mu$  on $\mathbb{R}^n$ is called log-concave if it has the form $e^{-V} \cdot \mathcal{H}^{k}|_{L}$, where $\mathcal{H}^k$ is the $k$-dimensional Hausdorff
measure, $k \in \{0,1, \cdots, n\}$, $L$ is an affine subspace,  and $V$ is a convex function. 

In what follows we consider uniformly log-concave measures. Roughly speaking, these  are the measures with potential $V$
satisfying 
$$
V(x) - V(y) - \langle \nabla V(y), x-y \rangle \ge \frac{K}{2}|x-y|^2,
$$ 
which is equivalent to $D^2 V \ge K \cdot \mbox{Id}$ in the smooth (finite-dimensional) case. Here $K$ is a positive constant.

More precisely, we say that a probability measure $\mu$ is $K$-uniformly log-concave  ($K > 0$) if for any $\varepsilon>0$ the measure $\hat{\mu} = \frac{1}{Z} e^{\frac{K-\varepsilon}{2} |x|^2} \cdot \mu$ is  log-concave  for a suitable renormalization factor $Z$.
It is well-known (C. Borell) that the projections  of log-concave measures are log-concave (this is in fact a corollary of the Brunn-Minkowski theorem). It can be easily checked that the uniform log-concavity is preserved by projections as well.
We can extend this notion to the infinite-dimensional case. Namely, we call a probability measure $\mu$ on a locally convex space $X$ log-concave ($K$-uniformly log-concave with $K>0$)  if its images  $\mu \circ l^{-1}$, $l \in X^*$ under linear continuous
functionals  are  all log-concave ($K$-uniformly log-concave with  $K>0$).

Throughout the paper we apply the following estimate (see  \cite{Kol04}, \cite{Kol2010}), which generalizes the famous Talagrand
transportation inequality.

\begin{theorem}
\label{generalTalagrand} ({\bf Generalized Talagrand inequality.})
Let $m$ be a $K$-uniformly log-concave probability measure with some  $K>0$. Then for any couple of probability measures  $\mu = e^{-V} \ dx$, $\nu = e^{-W} \ dx $
and the corresponding  optimal mappings $\nabla \varphi_{\mu}$, $\nabla \varphi_{\nu}$,
pushing forward $\mu$, $\nu$ onto $m$ respectively,
one has the following estimate
$$
\mbox{\rm {Ent}}_{\nu} \Bigl( \frac{\mu}{\nu} \Bigr)  = \int  \log \frac{d\mu}{d\nu} \ d \mu  = \int  (W-V) \ d \mu \ge
\frac{K}{2} \int \bigl( \nabla \varphi_{\mu} - \nabla \varphi_{\nu} \bigr)^2 \ d \mu.
$$
\end{theorem}

Another result used in the paper is the  Cafarelli's contraction theorem. Here is the version from \cite{Kol2010} (see also \cite{Kol-contr}).
\begin{theorem} {\bf (Caffarelli contraction theorem).}
Let $\nabla \Phi$ be the optimal transportation of the probability measure
$\mu = e^{-V} dx$ into $\nu = e^{-W} dx$. Assume that for some positive $c, C$ one has
$D^2 V \le C \cdot \rm{Id}$, $D^2 W \ge c \cdot \rm{Id}$. Then $\nabla \Phi$ is Lipschitz with 
$\| \nabla \Phi\|_{Lip} \le \sqrt{\frac{C}{c}}$.   
\end{theorem}

The quantity $\mbox{\rm {Ent}}_{\nu} \Bigl( \frac{\mu}{\nu} \Bigr)$ is called the
relative entropy or the Kullback-Leibler distance between $\mu$ and $\nu$.

\section{Sufficient condition for existence of limits of finite-dimensional optimal mappings}

\subsection{Preliminary finite-dimensional estimates}

Let $\mu$ and $\nu$ be probability measures on $\mathbb{R}^d$
and $T(x) = \nabla \varphi(x)$ be  the optimal transportation mapping pushing forward $\mu$ onto $\nu$.
Let us denote by $\mu_v$ the images of $\mu$ under the shifts  $ x \mapsto x + v$, $v \in \mathbb{R}^d$.

It will be assumed throughout that  $\mu_v$   have densities with respect to $\mu$:
$$
\frac{d \mu_v}{d \mu} = e^{\beta_v}.
$$

\begin{lemma}
\label{L2-D2}
For every $p, q \ge 1$ with $\frac{1}{p} + \frac{1}{q}=1$,  $\varepsilon \ge 0$,
and $ e \in \mathbb{R}^d$
$$
\int  |\varphi(x+ te) - \varphi(x)|^{1+\varepsilon} \ \ d \mu
\le  t^{1+\varepsilon} \|  \ |\langle x, e \rangle|^{1+\varepsilon} \|_{L^{p}(\nu)} \cdot \sup_{0 \le s \le t}  \| e^{\beta_{se}} \|_{L^{q}(\mu)}.
$$
$$
\int \bigl(\varphi(x+ te) - \varphi(x)  - t  \partial_e \varphi(x) \bigr) \ d \mu \le  t  \| \langle x, e \rangle \|_{L^p(\nu)} \cdot  \sup_{0 \le s \le t}  \| e^{\beta_{se}} -1 \|_{L^{q}(\mu)}.
$$
\end{lemma}
\begin{proof}
One has
$
\varphi(x+ te) - \varphi(x) = \int_{0}^{t} \partial_{e} \varphi(x + s e) \ ds.
$
Hence
\begin{align*}
\int &  |\varphi(x+ te) - \varphi(x)|^{1+\varepsilon} \ d \mu
\le t^{\varepsilon} \int \int_0^ t |\partial_e\varphi |^{1+\varepsilon} (x + s e) \ ds \ d \mu
\\& =   t^{\varepsilon}  \int_0^t \Bigl[ \int | \partial_e\varphi|^{1+\varepsilon}  e^{\beta_{se}} \ d \mu \Bigr] \ ds
 \le  t^{1+\varepsilon}   \| |\partial_e \varphi|^{1+\varepsilon}  \|_{L^{p}(\mu)} \cdot \sup_{0 \le s \le t}  \| e^{\beta_{se}} \|_{L^{q}(\mu)}
 \\& = t^{1+\varepsilon}  \| \  |\langle x, e \rangle|^{1+\varepsilon}  \|_{L^{p}(\nu)} \cdot \sup_{0 \le s \le t}  \| e^{\beta_{se}} \|_{L^{q}(\mu)}.
\end{align*}
Applying the same arguments one gets
\begin{align*}
\int & \bigl(\varphi(x+ te) - \varphi(x)  - t \partial_e \varphi(x) \bigr) \ d \mu = \int \int_0^t (\partial_e\varphi(x+se) -  \partial_e \varphi(x)) \ ds \ d \mu
\\&
= \int \Bigl[ \int_0^t (e^{\beta_{se}}-1) \ ds \Bigr] \partial_e \varphi(x) \ d\mu
\le
 t^{\frac{1}{p}} \| \partial_e \varphi \|_{L^p(\mu)} \Bigl[ \int \int_0^t  |e^{\beta_{se}}-1|^q \ ds \ d \mu \Bigr]^{\frac{1}{q}}.
\end{align*}
The desired estimate follows from the the change of variables formula and trivial uniform bounds.
\end{proof}

In addition, we will apply the following elementary Lemma.

\begin{lemma}
\label{conv-lemm}
Assume that a sequence $\{T_n\}$ of measurable mappings $T_n \colon \mathbb{R}^{\infty} \to  \mathbb{R}^{\infty}$  converges to a mapping $T$ in the following sense: for every $e_i$
$\lim_n \langle T_n, e_i \rangle  = \langle T, e_i \rangle$ in measure with respect to $\mu$. Then the measures $\{\mu \circ T^{-1}_n\}$ converge weakly to $\mu \circ T^{-1}$.
\end{lemma}

\subsection{Existence theorem} We consider a couple of Borel probability measures $\mu$ and $\nu$ on $\mathbb{R}^{\infty}$, where $\mathbb{R}^{\infty}$
is the space of all real sequences: $\mathbb{R}^{\infty} = \prod_{i=1}^{\infty} \mathbb{R}_i$.
We deal with the standard coordinate system $x = (x_1, x_2, \cdots, x_n, \cdots)$ and the standard basis vectors $e_i = (\delta_{ij})$. The projection
on the first $n$ coordinates will be denoted by $P_n$:
$P_n(x) = (x_1, \cdots, x_n)$. We use  notations $\| x \|$, $\langle x, y \rangle$ for the
Hilbert space norm and inner product: $\|x\| = \sum_{i=1}^{\infty} x^2_i$, $\langle x, y \rangle = \sum_{i=1}^{\infty} x_i y_i$.
We use notation $\Exp^{n}_{\mu}$ for the conditional expectation with respect to $\mu$ and the $\sigma$-algebra generated by $x_1, \cdots, x_n$.
For any product measure  $P = \prod_{i=1}^{\infty} p_i(x_i) \ dx_i$ its projection $P_n = P \circ P^{-1}_n$   has the form  $ \prod_{i=1}^{n} p_i(x_i) \ dx_i$  and
the projection $(f \cdot P) \circ P^{-1}_n = f_n \cdot  P_n$ of the measure $f \cdot P$ satisfies $f_n = \Exp^{n}_{P} f$.
Everywhere below we agree that every cylindrical function $f = f(x_1, \cdots, x_n)$ can be extended to $\mathbb{R}^\infty$
by the formula $x \to f_n(P_n x)$.

It will be assumed throughout the paper  that the shifts of $\mu$ along any vector $v = t e_i$  are absolutely continuous with respect to $\mu$:
$$
\frac{d \mu_v}{d \mu} = e^{\beta_v}.
$$
In Section 3,  moreover, the following assumption holds.

{\bf Assumption (A).} For every basic vector  $e=e_i$ there exist $p \ge 1$, $q  \ge 1$, satisfying $\frac{1}{p} + \frac{1}{q} =1$, and $\varepsilon>0$  such that
$$ \int |\langle x, e \rangle|^{(1+\varepsilon)p} \ d \nu < \infty
$$
and
$$
 p(t) =  \sup_{0 \le s \le t}   \int |e^{\beta_{se}}-1|^q \ d \mu
$$
satisfies $\lim_{t \to 0} p(t) =0$.

Let $\mu_n = \mu \circ P^{-1}_n(x)$, $\nu_n = \nu \circ P^{-1}_n(y)$ be the projections of $\mu$, $\nu$.
For every $v = t e_i$ let us set
$$
\frac{d (\mu_n)_v}{d \mu_n} = e^{\beta^{(n)}_{v}}.
$$
It is easy to check that  the projections of  $\mu, \nu$  satisfy Assumption (A).
\begin{lemma}
\label{proj}
For every $n \in \mathbb{N}$ and every  $e=e_i$ one has
$$
 \int |\langle P_n(x), e \rangle|^p \ d \nu_n \le \int |\langle x, e \rangle|^p \ d \nu, \ \ \  \int |e^{\beta^{(n)}_{e}}-1|^q \ d \mu_n \le   \int |e^{\beta_{e}}-1|^q \ d \mu.
$$
\end{lemma}
\begin{proof}
The first estimate is trivial. To prove the second one, let us note that
$ e^{\beta^{(n)}_{v}} =   \Exp^{n}_{\mu}  e^{\beta_{v}} $. The claim follows from the Jensen inequality
and convexity of the function $t \to |t-1|^q$.
\end{proof}

We denote by $\pi_n$ the optimal transportation plan for the couple $(\mu_n, \nu_n)$.
Let $\varphi_n(x)$ and $\psi_n(y)$ solve the dual Kantorovich problem.
Let us recall  that
$\nabla \varphi_n$ ($\nabla \psi_n$) is the  optimal transportation mapping sending  $\mu_n$ to $\nu_n$  ($\nu_n$ to $\mu_n$).
One has
$$
\varphi_n(x) + \psi_n(y) \ge \langle P_n x, P_n y\rangle
$$
for every $x, y$.
The equality is attained on the support of $\mathbb{\pi}_n$. In particular,
$$
\varphi_n(x) + \psi_n(\nabla \varphi_n(x))  =  \langle P_n x,   \nabla \varphi_n(x)\rangle.
$$

It is easy to check that $\{ \pi_n\}$ is a tight sequence.
By the Prokhorov theorem one can extract a weakly convergent subsequence $\pi_{n_k}\to \pi$.
Note that $\pi_n$  {\bf is not} the projection of $\pi$.

The main result if the section is the following theorem.
\begin{theorem}
\label{existence-crit}
Assume that {\bf (A)} is fulfilled and, in addition,
 $$F_n(x,y,0,0)  = \varphi_n(x) + \psi_n(y) - \langle P_n x, P_n y\rangle  \to 0
 $$ in measure with respect to $\pi$. Then
there exists a mapping $T \colon \mathbb{R}^{\infty} \mapsto \mathbb{R}^{\infty}$ such that
$$
T(x)=y
$$
for $\pi$-almost all $(x,y)$.
\end{theorem}

In what follows we will pass several time to  subsequences and use for the  new subsequences the same index
$n$ again,  with the agreement that $n$ takes values in another infinite set $ \mathbb{N}' \subset \mathbb{N}$.
Let us fix  unit vectors $e_i, e_j$ for some $i, j \in \mathbb{N}$ and  consider the following sequence of non-negative functions:
$$
F_n(x,y,t,s) = \varphi_n(x + t e_i)  + \psi_n(y + s e_j) - \langle P_n (x + t e_i), P_n(y + s e_j) \rangle
$$
with $n > i, n >j$.

\begin{lemma}
\label{06.01.13}
There exists a $L^{1+\varepsilon}(\pi)$-weakly convergent subsequence
$$\varphi_{n_k}(x+ t e_i) - \varphi_{n_k}(x) \to U(x).$$
The following relation holds for the  limiting function $U(x)$:
$$
\Bigl| \int  U(x) \ d \mu - t \int \langle y, e_i \rangle \ d \nu \Bigr| \le C t p(t).
$$
\end{lemma}
\begin{proof}
Taking into account that $\int F_n(x,y,0,0) \ d \pi_n =0$, one obtains
$$
\int F_n(x,y,t,0) \ d \pi_{n} = \int F_n(x,y,t,0) \ d \pi_n - \int F_n(x,y,0,0) \ d \pi_n \ge 0.
$$
Note  that the right-hand side equals
$$
\int ( F_n(x,y,t,0) - F_n(x,y,0,0)) \ d \pi_{n} =
\int  \bigl[ \varphi_n(x+t e_i) - \varphi_n(x)  - t \langle y, e_i \rangle \bigr]   \ d \pi_n.
$$
Taking into account that the projection of $\pi_n$  onto $X$ coincides with $\mu_n$ and $\varphi_n$ depends on the  first $n$ coordinates, one finally obtains that for $n > i$ the latter is equal to
$$
\int  \bigl[ \varphi_n(x+t e_i) - \varphi_n(x) \bigr] \ d \mu    - t \int  \langle y, e_i \rangle \ d \nu = \int  \bigl[ \varphi_n(x+t e_i) - \varphi_n(x)  - t \partial_{e_i} \varphi_n(x) \bigr] \ d \mu.
$$
It follows from  Lemma \ref{L2-D2}, Lemma \ref{proj} and Assumption (A) that
\begin{equation}
\label{04.10}
\Bigl| \int F_n(x,y,t,0) \ d \pi_{n}  \Bigr| \le C t p(t).
\end{equation}

Since $\varphi_n$ depends on a finite number of coordinates ($\le n$), one has $$\int |\varphi_n(x+t e_i) - \varphi_n(x) |^{1+\varepsilon} \ d \mu = \int |\varphi_n(x+t e_i) - \varphi_n(x) |^{1+\varepsilon} \ d \mu_n.$$
Hence by Lemma \ref{L2-D2}
$$
U_n(x) = \varphi_n(x+t e_i) - \varphi_n(x) \in L^{1+\varepsilon}(\mu)
$$
and, moreover, $\sup_n  \| U_n\|_{L^{1+\varepsilon}(\mu)}  < \infty$. Thus there exists function $U \in L^{1+\varepsilon}(\mu)$ such that
for some subsequence $n_k$
$$
\varphi_{n_k}(x+t e_i) - \varphi_{n_k}(x) \to U(x)
$$
weakly in $L^{1+\varepsilon}(\mu)$.
Passing to the limit we obtain from (\ref{04.10}) that
$$
\Bigl| \int  U(x) \ d \mu - t \int \langle y, e_i \rangle \ d \nu \Bigr| \le  C t p(t).
$$
\end{proof}

\begin{lemma}
\label{06.01.13(2)}
Assume that $F_n(x,y,0,0) \to 0$ in measure with respect to $\pi$. Then
$$
U(x) - t\langle y, e_i \rangle \ge 0
$$
for $\pi$-almost all $(x,y)$.
\end{lemma}
\begin{proof}
Note that
$$
\bigl[ \varphi_n(x+t e_i) - \varphi_n(x) - t\langle y, e_i \rangle  \bigr] + F_n(x,y,0,0) =  \varphi_n(x+t e_i)  + \psi_n(y) - \langle P_n y, P_n(x + t e_i) \rangle
$$
is a non-negative function for every $n$.
Since $F_n(x,y,0,0) \to 0$ in measure, there exists a subsequence (denoted again by $F_n$)
which converges to zero  $\pi$-almost everywhere.
Since  $f_n = \varphi_n(x+t e_i) - \varphi_n(x) - t\langle y, e_i \rangle$ converges to $ f = U(x) - t\langle y, e_i \rangle$ weakly in $L^{1+\varepsilon}(\pi)$, one can assume (passing again to a subsequence)
that $\frac{1}{N} \sum_{n=1}^{N} f_n \to f $ $\pi$-a.e. Since $f_n + F_n \ge 0$, this implies that $f \ge 0$ $\pi$-a.e.
\end{proof}

\begin{proposition}
\label{monotone-exist}
Assume that there exists a sequence of continuous functions $$f_n(x_1, \cdots, x_n), g_n(y_1, \cdots, y_n) \in L^1(\pi_n)$$ such that $G_n = f_n(x) + g_n(y) - \sum_{i=1}^n x_i y_i$ has the following properties:
\begin{itemize}
\item[1)]
$G_n \ge 0$,
\item[2)]
$
G_n \le G_m, \ \ \forall  \  n \le m, x,y \in \mathbb{R}^m,
$
\item[3)]
$
\sup_n \int
G_n \ d \pi_n < \infty.
$
\end{itemize}
Then $F_n(x,y,0,0) \to 0$ in $L^1(\pi)$.
\end{proposition}
\begin{proof}
We start with the identity $\int F_n(x,y, 0, 0) \ d \pi_n =0$ and rewrite it in the following way:
\begin{equation}
\label{05.01.13}
0 = \int ( \varphi_n - f_n) \ d \mu +  \int( \psi_n - g_n) \ d \nu
+ \int \bigl( f_n(x) + g_n(y) - \sum_{i=1}^n x_i y_i \bigr) \ d \pi_n.
\end{equation}
Since $\varphi_n, \psi_n$ are defined up to a constant, one can assume that $\int( \psi_n - g_n) \ d \nu=0$.
Thus $
- \int ( \varphi_n - f_n) \ d \mu = \int \bigl( f_n(x) + g_n(y) - \sum_{i=1}^n x_i y_i \bigr) \ d \pi_n.
$
It follows from 1) and 3) that the right-hand side is a bounded sequence of non-negative numbers.
Passing to a subsequence we may assume that  the right-hand side has a limit.
It follows from the weak convergence $\pi_n \to \pi$ and the monotonicity property 2) that for every $k$
\begin{align*}
\underline{\lim}_n  & \int \bigl( f_n(x) + g_n(y) - \sum_{i=1}^n x_i y_i \bigr) \ d \pi_n
\ge
\underline{\lim}_n  \int \bigl( f_k(x) + g_k(y) - \sum_{i=1}^k x_i y_i \bigr) \ d \pi_n
\\&
= \int \bigl( f_k(x) + g_k(y) - \sum_{i=1}^k x_i y_i \bigr) \ d \pi.
\end{align*}
Hence
\begin{align*}
\underline{\lim}_n &  \int \bigl( f_n(x) + g_n(y) - \sum_{i=1}^n x_i y_i \bigr) \ d \pi_n
\ge
\lim_k \int \bigl( f_k(x) + g_k(y) - \sum_{i=1}^k x_i y_i \bigr) \ d \pi,
\end{align*}
 where the limit in the right-hand side exists, because the sequence is  monotone.
Hence we get from (\ref{05.01.13})
$$
0 \ge \lim_n  \int ( \varphi_n - f_n) \ d \mu  + \lim_n \int   \bigl( f_n(x) + g_n(y) - \sum_{i=1}^n x_i y_i \bigr) d \pi.
$$
Taking into account that $\int g_n \ d \pi = \int g_n \ d \nu = \int \psi_n \ d \nu = \int \psi_n \ d \pi $, we obtain
\begin{align*}
0 & \ge \lim_n  \int ( \varphi_n - f_n)(x) \ d \mu  + \lim_n \int   \bigl( f_n(x) + g_n(y) - \sum_{i=1}^n x_i y_i \bigr) d \pi	
\\& = \lim_n \Bigl( \int ( \varphi_n(x) + \psi_n(y) - \sum_{i=1}^n x_i y_i)  \   d \pi \Bigr) \ge 0.
\end{align*}
The proof is complete.
\end{proof}

Finally,  we obtain a sufficient condition for the
 existence of an optimal mapping in the infinite-dimensional case.

\begin{proof} {(\bf Theorem \ref{existence-crit})}
Let us fix $e_i$ and choose a sequence of numbers $t_n \to 0$.
We get from Lemma \ref{06.01.13} and Lemma \ref{06.01.13(2)} that there exist  $\pi$-a.e. nonnegative functions $U_{t_n}(x) - t_n \langle y, e_i \rangle $ with
$\int \bigl( U_{t_n}(x) - t_n\langle y, e_i \rangle \bigr) \ d \pi =o(t_n)$. Hence, $\lim_{t_n \to 0} \int \bigl( \frac{U_{t_n}(x)}{t_n} - \langle y, e_i \rangle \bigr) \ d \pi =0$.
Taking into account that $\frac{U_{t_n}(x)}{t_n} - \langle y, e_i \rangle \ge 0$ for $\pi$-almost all $(x,y)$, we conclude that $\frac{U_{t_n}(x)}{t_n}$ converges $\mu$-a.e. and in $L^1(\mu)$ to a function
$u_i(x)$ satisfying $u_i(x) - \langle y, e_i \rangle \ge 0$, $\pi$-a.e. and $\int (u_i(x) - \langle y, e_i \rangle) \ d \pi =0$.  Clearly, $u(x) = \langle y, e_i \rangle$ for $\pi$-almost all $(x,y)$. Repeating these arguments
for every $i \in \mathbb{N}$, we get the claim.
\end{proof}

\section{Application: quasi-product case}

The main result of this section is a generalization of the optimal transport existence theorem for Gaussian measures.
Recall that by results from \cite{FU1}, \cite{Kol04} that for the 
standard Gaussian measure $\gamma = \prod_{i=1}^{\infty} \gamma_i(dx_i)$, $\gamma_i \sim \mathcal{N}(0,1)$
the existence of the optimal transportation mapping  pushing forward $f \cdot \gamma$ into $g \cdot \gamma$ is established, for instance,
under assumption $\int f \log f \ d \gamma < \infty, \int g \log g \ d \gamma < \infty$.
We give in this section a generalization of this result for a wide class of quasi-product measures.

Let  us consider  two product reference  measures   $$P = \prod_{i=1}^{\infty} p_i(x_i) \ dx_i , \ \ Q = \prod_{i=1}^{\infty} q_i(x_i) dx_i$$
and fix the diagonal infinite transportation mapping 
$$
T(x) = (T_1(x_1), \cdots, T_n(x_n), \cdots)
$$
 where $T_i(x_i)$ pushes forward $p_i(x_i)  dx_i$ onto $q_i(x_i)  dx_i$.
Clearly, $T$ takes $P$ onto $Q$.
The inverse mapping $S=T^{-1}$ has the same diagonal structure: $$
S(x) = (S_1(x_1), \cdots, S_n(x_n), \cdots).
$$

\begin{theorem}
Let  $\mu = f \cdot P$ and $\nu = g \cdot Q$ be probability measures
 satisfying the Assumption (A) of the previous section. Assume, in addition, that
\begin{itemize}
\item[1)]
there exists $K>0$ such that
every $q_i$ is $K$-uniformly log-concave;
\item[2)]
there exists $M>0$ such that
$$
 S'_i(x_i) \le M;
$$
for all $i, x_i$;
\item[3)]
Assume that either a) or b) holds for some constants $C>c>0$
\begin{itemize}
\item[a)] $g \log^2 g \in L^1(Q), \ \frac{1}{f} \in L^1(P) , \ f \le C$,
\item[ b)] $
f \log f \in L^1(P), \ \ c \le g \le C$.
\end{itemize}
\end{itemize}
Then there exists a transportation mapping $T$ pushing forward $\mu$ onto $\nu$ which is $\mu$-a.e. limit
of finite-dimensional optimal transportation mappings $T_n$.
\end{theorem}
\begin{remark}
It follows from  Caffarelli's contraction theorem (see Section 2)
that assumption 2) is satisfied if $(-\log p_i(x_i))'' \ge C_0$, $(-\log q_i(x_i))'' \le C_1$  for some $C_0, C_1 >0$ and every $i$. Of course, there exist
 many other examples when this assumption is satisfied.
\end{remark}
\begin{proof}
Consider the finite-dimensional projections $\mu_ n = f_n \cdot P_n$, $\nu_n = g_n \cdot Q_n$, where $P_n = \prod_{i=1}^{n} p_i(x_i) \ dx_i$, $Q_n = \prod_{i=1}^{n} q_i(x_i) \ dx_i$.
Here $f_n$ and $g_n$ are the conditional expectations of $f, g$ with respect to $P, Q$
and the $\sigma$-algebra $\mathcal{F}_n$, generated by the first $n$ coordinates.
Recall that $\nabla \varphi_n$ is the optimal transportation of $\mu_n$ to $\nu_n$.
Let $$u_i(x_i),  \ v_i(y_i) = u_i^*$$ be the  one-dimensional convex
potentials associated to the mappings $T_i, S_i$, respectively: $$T_i=u'_i, \ S_i=v'_i.$$ Note that $\tilde{T}_n  = (T_1, \cdots, T_n)$ pushes forward
 $ P_n$ onto $Q_n$
and
$\nabla \varphi_n$ pushes forward
$\frac{f_n}{g_n(\nabla \varphi_n)} \cdot P_n$ onto $Q_n$.

According to Proposition \ref{generalTalagrand} one has the following estimate:
\begin{equation}
\label{03.03.13(1)}
\frac{K}{2} \int |\tilde{T}_n - \nabla \varphi_n| ^2  d P_n
\le \int \log \Bigl( \frac{g_n(\nabla \varphi_n)}{f_n} \Bigr)   d P_n.
\end{equation}
To see that the right-hand side is finite, let us estimate
\begin{align*}
\int \log \Bigl( \frac{g_n(\nabla \varphi_n)}{f_n} \Bigr)   d P_n
\le & \int \log \frac{1}{f_n}   d P_n
+
\frac{1}{2}\int \log^2 g_n(\nabla \varphi_n) f_n d P_n + \frac{1}{2}  \int \frac{dP_n}{f_n}
\\&
= \int \log \frac{1}{f_n}   d P_n
+
\frac{1}{2}  \int g_n  \log^2 g_n  d Q_n + \frac{1}{2}  \int \frac{dP_n}{f_n}.
\end{align*}
Applying  Assumption 3a of the Theorem and  the Jensen inequality one can easily get that
the right-hand side is uniformly bounded.

We complete the proof by applying   Theorem \ref{existence-crit} and Proposition \ref{monotone-exist}.
For application of Proposition \ref{monotone-exist}
set
$$ f_n = \sum_{i=1}^n u_i(x_i),  \ g_n =\sum_{i=1}^n v_i(y_i).$$
We need to  estimate
$
\sum_{i=1}^n \int (u_i(x_i) + v_i(y_i) -  x_i y_i) \ d\pi_n.
$
Taking into account that $\pi_n$ is supported on the graph of $\nabla \varphi_n$, and the relation $u_i(x_i) + v_i(T_i(x)) = x_i T_i(x)$ we obtain that the latter equals to
\begin{align*}
\int & \bigl( u_i(x_i) + v_i(\partial_{x_i}\varphi_n) -  x_i \partial_{x_i} \varphi_n(x) \bigr) \ d \mu_n
\\&
=
\int  \Bigl[ v_i(\partial_{x_i} \varphi_n(x)) - v_i(T_i(x)) -  x_i (\partial_{x_i} \varphi_n(x) - T_i(x)) \Bigr] \ d \mu_n
\\& =
\int  \Bigl[ v_i(\partial_{x_i} \varphi_n(x)) - v_i(T_i(x)) -  v'_i(T_i(x)) (\partial_{x_i} \varphi_n(x) - T_i(x)) \Bigr] \ d \mu_n
\\&
\le M \int ( \partial_{x_i} \varphi_n(x) - T_i)^2 \ d \mu_n.
\end{align*}
Here we use the  uniform bound $v''_i = S'_i \le M$.
Finally, using the uniform bound $f \le C$ and the Jensen inequality we obtain that
\begin{align*}
\sum_{i=1}^n & \int (u_i(x_i) + v_i(y_i) -  x_i y_i) \ d\pi_n \le M C \int |\nabla \varphi_n - \tilde{T}_n|^2 \ d P_n.
\end{align*}
We have already shown that the right-hand side is bounded.
The result now follows from  Proposition \ref{monotone-exist}.

The proof follows the same line under Assumption 3b, but we use 
 another corollary of Proposition \ref{generalTalagrand}:
$$
\frac{K}{2} \int |\tilde{T}_n - \nabla \varphi_n| ^2 \frac{f_n}{g_n(\nabla \varphi_n)}  d P_n
\le \int \log \Bigl( \frac{f_n}{g_n(\nabla \varphi_n)} \Bigr) \frac{f_n}{g_n(\nabla \varphi_n)}  d P_n.
$$
The detailes are left to the reader.
\end{proof}

\section{Symmetric transportation problem and ergodic decomposition of optimal transportation plans}

\subsection{Symmetric transportation problem}

In this section we discuss the mass transportation of 
symmetric (mainly exchangeable) measures, where the word ''symmetric'' means ''invariant under action of a group $\Gamma$''.

Recall that a probability measure is exchangeable
 if it is  invariant with respect to any permutation of finite number of coordinates.
 Before we consider $\mathbb{R}^{\infty}$, let us make some remarks
 on the finite-dimensional case.

Consider  the group  $S_{d}$ of all permutations of $\{1, \cdots , d\}$ acting on $\mathbb{R}^{d}$ as follows:
$$
L_{\sigma}(x) = (x_{\sigma(1)}, x_{\sigma(2)}, \cdots, x_{\sigma(d)}), \ \ \ \sigma \in S_{d}.
$$
Let $\Gamma \subset S_d$ be any subgroup with the property that
for every couple $i, j$ there exists $\sigma \in \Gamma$ such that $\sigma(i)=j$.

Assume that the source and target measures are both invariant with respect to $\Gamma$. 
Under additional assumption that the cost function $c$ is $\Gamma$-invariant (for instance, $c=|x-y|^2$) one can easily check that the Kantorivich potential $\varphi$ is $\Gamma$-invariant as well:
$\varphi = \varphi \circ L_{\sigma}$ for any $\sigma \in \Gamma$ {see \cite{Moameni}, \cite{Zaev1}}. Consequently, the optimal transportation $T = \nabla \varphi$ has the following commutation property:
$$
T = L^{*}_{\sigma} (  T \circ L_{\sigma})
=
 L^{-1}_{\sigma} \circ T \circ L_{\sigma}.
$$
Equivalently,
$$
  L_{\sigma} \circ T = T \circ L_{\sigma} .
$$
The optimal transportation  plan $\pi(dx, dy)$ is also $\Gamma$-invariant
under the following extension of the action of $\Gamma$ to $\mathbb{R}^d \times \mathbb{R}^d$: 
$$L_{\sigma}(x,y) =(L_{\sigma}x, L_{\sigma} y). 
$$

Now let $\sigma(i)=j$. One has
\begin{align*}
\int x_i y_i \ d \pi & =
 \int \langle e_i,  x\rangle  \langle e_i, y \rangle \ d \pi  =
\int \langle L_{\sigma} e_i,  L_{\sigma} x \rangle \langle L_{\sigma} e_i,  L_{\sigma} y \rangle\ d \pi
\\&   = \int \langle e_j,  L_{\sigma} x \rangle \langle e_j, L_{\sigma} y \rangle \ d \pi
=
  \int x_j y_j \ d \pi.
\end{align*}
Consequently,
\begin{equation}
\label{W22dint}
W^2_2(\mu,\nu) = \int \|x-y\|^2 \ d \pi =  \sum_{i=1}^d \int (x_i - y_i)^2 \ d \pi = d \int (x_i - y_i)^2 \ d \pi, \ \ \forall i.
\end{equation}

\begin{lemma} 
\label{equiv}
The standard quadratic Kantorovich
problem on $\mathbb{R}^d$ with $\Gamma$-invariant marginals is equivalent to the transportation problem for the cost $|x_1-y_1|^2$ 
with additional constraint that the solution is a  $\Gamma$-invariant probability  measure
\end{lemma}
\begin{proof}
Let $\pi$ be the solution to the quadratic Kantorovich problem for the marginals $\mu, \nu$ 
and $\tilde{\pi}$ be a measure giving the minimum to the functional
$m \mapsto \int |x_1-y_1|^2 d m$ among of the $\Gamma$-invariant measures with the same marginals.
By optimality of $\pi$
$$
\int |x-y|^2 d \pi \le \int |x-y|^2 d \tilde{\pi}.
$$
Since $\pi$ and $\tilde{\pi}$ are both $\Gamma$-invariant, (\ref{W22dint}) implies that $\int |x_1-y_1|^2 d \pi \le \int |x_1-y_1|^2 d \tilde{\pi}.$  By optimality of $\tilde{\pi}$ one gets $\int |x_1-y_1|^2 d \pi = \int |x_1-y_1|^2 d \tilde{\pi}$, and, finally
$\int |x-y|^2 d \pi = \int |x-y|^2 d \tilde{\pi}$. This means that $\tilde{\pi}$ solves the quadratic Kantorovich problem as well and, vice versa,
$\pi$ solves the Kantorovich problem with symmetric constraints.
\end{proof}

The conclusion made above helps us to give a variational meaning to the transportation problem in the infinite-dimensional case.
\begin{definition}
\label{symMK-def} {\bf Symmetric Kantorovich problem.}
Let $\Gamma$ be a group of linear operators acting on $\mathbb{R}^{\infty}$ and $\mu,\nu$ be
$\Gamma$-invariant probability measures.
Assume in addition that
\begin{itemize}
\item
For every $i, j \in \Nat$ there exists $g \in \Gamma$ such that
$$
g(e_i)=e_j.
 $$ 
 \item
 The space of probability measures $\Pi^{\Gamma}(\mu,\nu)$ on $\mathbb{R}^{\infty} \times \mathbb{R}^{\infty}$
 which are invariant with respect to the action
 $(x,y) \mapsto (g(x),g(y))$, $g \in \Gamma$ of  $\Gamma$
 and have marginals $\mu,\nu$, is non-empty and closed in the weak topology.
\end{itemize}
We say that a measure $\pi \in  \Pi^{\Gamma}(\mu,\nu)$  is a solution
 to the  $\Gamma$-symmetric (quadratic)  Kantorovich problem
if it gives the minimum to the functional
\begin{equation}
\label{symMK}
\Pi^{\Gamma}(\mu,\nu)\ni m \mapsto \int (x_1-y_1)^2 \ d m.
\end{equation}
\end{definition}
\begin{definition} {\bf Symmetric optimal transportation.}
Let $m$ be a solution to the symmetric Kantorovich problem.
A measurable mapping $T \colon \mathbb{R}^{\infty} \mapsto \mathbb{R}^{\infty}$ is called  optimal transportation mapping of $\mu$ onto $\nu$ if
$$m(\{(x,T(x))\})=1.$$
\end{definition}

The standard compactness arguments imply that a solution to the Kantorovich problem (\ref{symMK}) exists provided $\int x^2_1 \ d \mu < \infty, \int y^2_1 \ d \nu < \infty$.
If, in addition, there exists an optimal transportation mapping $T$, it commutes with any $g \in \Gamma$. This means that for $\mu$-almost all $x$ and every $g \in \Gamma$ 
\begin{equation}
\label{commute}
T \circ g (x) = g \circ T(x).
\end{equation}

\begin{example} {\bf Exchangeable measures.}
\label{ExchM}
We denote by  $S_{\infty}$ the group of permutation of $\mathbb{N}$ which change only a finite number of coordinates. We consider its natural action on $\mathbb{R}^{\infty}$ defined by
$$
\sigma(x) = (x_{\sigma(i)}), \  \ x = (x_i) \in \mathbb{R}^{\infty}, \ \ \sigma \in S_{\infty}.
$$
Consider measures $\mu$ and $\nu$ which are invariant with respect to any   $\sigma \in S_{\infty}$:
$$
\mu = \mu \circ \sigma^{-1}, \ \
\nu = \nu \circ \sigma^{-1}.
$$
The measures of this type are called exchangeable. The basic example is given by  the countable power $m^{\infty}$
of some Borel measure $m$ on $\mathbb{R}$.
The structure of mappings satisfying (\ref{commute}) in the case $\mu = m^{\infty}$ is very easy to describe. Consider the function $T_1(x) = \langle T(x), e_1 \rangle$ and fix the first coordinate $x_1$. Then the function $F \colon (x_2, x_3, \cdots) \to T_1(x)$
is invariant with respect to $S_{\infty}$ (acting on $(x_2,x_3, \cdots)$). Hence $F$ is constant
according by the Hewitt--Sawage $0-1$ law applied to the measure $\mu$. Thus $T_1(x) = T_1(x_1)$ depends on
$x_1$ only (up to a set of measure zero). The same arguments applied to other coordinates imply that  $T$ is diagonal: $(T_1(x_1), T_2(x_2), \cdots)$. Moreover, $T_i(x)=T_1(x)$ because $T$ commutes with every permutation of coordinates.
\end{example}

\begin{example}
\label{noexist-average} {\bf Optimal transportation not always exists.}
Let $\mu_1, \mu_2$ be countable powers of two different one-dimensional measures.
By the Kakutani dichotomy theorem they are mutually singular.
There is no any mass transportation $T$ of $\mu = \mu_1$ onto $\nu=\frac{1}{2} (\mu_1 + \mu_2)$ satisfying (\ref{commute}).
Indeed, according to Example \ref{ExchM} any $T$ satisfying (\ref{commute}) must be diagonal, hence the measure $\mu \circ T^{-1}$ must be a product measure.
\end{example}

Thus, we see that the optimal transportation does  not always exist.
This example can be easily generalized to many other linear groups $\Gamma$
and $\Gamma$-invariant measures.  It can be easily understood that $T$ does not exists
provided the source measure is ergodic, but the target measure is not.

\subsection{Ergodic decomposition of optimal transportation plans}

The connection between Kantorovich problem and ergodic decomposition has been established 
under fairy general assumptions by the second-named author in \cite{Zaev2}. 
A particular case of this result is given in the following theorem.

Let $\Gamma$ be an amenable group acting  by continuous one-to-one mappings on a Polish space $X$. Let $\Pi^{\Gamma}$ be the set of all Borel probability
$\Gamma$-invariant measures  and $\mu, \nu \in \Pi^{\Gamma}$. The set of $\Gamma$-invariant
transportation plans with marginals $\mu,\nu$ will be denoted by $\Pi^{\Gamma}(\mu,\nu)$.  Assume that the cost function $c$ is  lower semicontinuous
and $\Pi^{\Gamma}(\mu,\nu)$ is non-empty and closed in the weak topology.

Let us fix a solution $\pi$ to the $\Gamma$-invariant Kantorovich problem with marginals $\mu,\nu$.
Denote by $\Delta(X)$ the set all $\Gamma$-invariant ergodic measures on $X$.
Assume we are given ergodic decompositions
\begin{equation}
\label{munudecomp}
\mu  = \int_{\Delta(X)}  \mu^x \ d {\sigma_{\mu}}, \  \nu  = \int_{\Delta(Y)}  \nu^y \ d {\sigma_{\nu}}
\end{equation}
of $\mu,\nu$, where $X=Y$, $\sigma_{\mu}, \sigma_{\nu}$ are probability measures on $\Delta(X), \Delta(Y)$
and, similarly, the ergodic decomposition  of $\pi$:
\begin{equation}
\label{pidecomp}
\pi = \int_{\Delta(X\times Y) } \pi^{x,y} d \delta
\end{equation}
(recall that the $\Gamma$-invariance for $\pi$ means the invariance with respect to the action $(x,y) \mapsto (g(x), g(y))$).
It is straightforward that $\delta$-almost all $\pi^{x,y}$ have ergodic marginals
and taking the projections of the both sides of (\ref{pidecomp}) we obtain decompositions 
(\ref{munudecomp}). Moreover, the following statement holds:
\begin{theorem}
\label{plandecom}
For $\delta$-almost all $(x,y)$ measure $\pi^{x,y}$ solves  the $\Gamma$-symmetric Kantorovich problem with  marginals
$\mu^x, \nu^y$:
$$
K^{\Gamma}_c(\mu^x,\nu^y) = \inf_{m \in \Pi^{\Gamma}(\mu^x,\nu^y)} \int c d m = \int c d \pi^{x,y}
$$
and the following representation formula holds:
$$
\inf_{\pi \in \Pi^{\Gamma}(\mu,\nu)} \int c d \pi = \inf_{\delta \in \Pi(\sigma_{\mu}, \sigma_{\nu})}
\int  
K^{\Gamma}_c(\mu^{x}, \nu^{y} ) \ d \delta. 
$$
\end{theorem}
\begin{remark}
In the situation of Theorem \ref{plandecom} one can decompose the optimal transportation plan for ergodic marginals $\mu,\nu$:
$\pi = \int_{\Delta(X\times Y) } \pi^{x,y} d \delta$. Ergodicity of the marginals implies immediately that $\delta$-almost all
$\pi^{x,y}$ have the same marginals $\mu$ and $\nu$. The optimality of $\pi^{x,y}$ for the cost $c$ follows from  Theorem \ref{plandecom}.
Thus we get that any solvable symmetric Kantorovich problem with ergodic marginals admits, in particular, an ergodic solution.
\end{remark}

Thus the symmetric transportation problem can be rediced to the following steps:
\begin{itemize}
\item[Q1)]
Construct a solution to the symmetric Kantorovich problem for ergodic measures.
\item[Q2)]
Given two non-ergodic measures $\mu,\nu$ and the corresponding ergodic decompositions
(\ref{munudecomp}) construct a solution to the  Kantorovich problem to measures $\sigma_{\mu}, \sigma_{\nu}$
on $\Delta(X)$ with the cost function $K^{\Gamma}_c$.
\end{itemize}

Consider application of Theorem \ref{plandecom} to several
classical groups.

\begin{example}
{\bf Exchangeable measures revisited.} Consider invariant transportation problem for exchangeable measures and $c=(x_1-y_1)^2$. The answer to Q1) is trivial, because ergodic measures are countable powers and
the structure of the corresponding solution is trivial. As for Q2), by the de Finetti theorem 
the space of ergodic measures is isomorphic to
the space  $\mathcal{P}(\mathbb{R})$ of probability measures on $\mathbb{R}$. 
Thus to resolve an optimal transportation problem for exchangeable measures, we need
to study the optimal transportation problem
for a couple of measures $\mu_0, \nu_0$
on $\mathcal{P}(\mathbb{R})$ arising from the de Finetti decomposition. It is clear that the  cost function $c$ on $\mathcal{P}(\mathbb{R})$
satisfies
$$
c(p_1, p_2) = W^2_2(p_1,p_2),
$$
where $W_2$ is the standard Kantorovich distance on $\mathbb{R}$.
\end{example}

\begin{example}
\label{Rim}
{\bf Rotationally invariant measures.}
Consider invariant transportation problem for measures invariant with respect to operators of the type $U \times Id$, where $U$ is a rotation of $\mathbb{R}^n = Pr_n (\mathbb{R}^{\infty})$ and $Id$ is the identical operator
on the orthogonal complement to $\mathbb{R}^n$ As usual $c=(x_1-y_1)^2$.
This is an example where the optimal transportation problem admits a precise solution.
By a well known result (see \cite{Kallenberg})
every rotationally invariant measure $\mu$ on $\mathbb{R}^{\infty}$
admits a representation
$$
\mu  = \int \gamma_t d p_{\mu}(t),
$$
where $\gamma_t$ is the distribution of the Gaussian i.i.d. with zero mean and variance $t$ and
$p_{\mu}$ is a measure on $\mathbb{R}_{+}$.
The optimal transportation problem is reduced obviously to the one-dimensional optimal 
transportation between $p_{\mu}$ and $p_{\nu}$.
\end{example}

\begin{example}
{\bf Stationary measures.} These are the measures which are invatiant with respect to the shift:
$$
T \colon x = (x_1, x_2, \cdots) \mapsto (x_2, x_3, \cdots).
$$
Note that the powers of $T$ generates the semigroup $\{0\} \cup \Nat$, but not the group.
However, it makes no difference for our analysis, we are still able to consider the corresponding ergodic decompositions.
In this case the description of ergodic measures is nontrivial and we do not know any general sufficient conditions for existence
even in the case when both measures are  ergodic.
Some sufficient conditions are given in Section 7.
\end{example}

We conclude the section with the remark that existence of a transportation mapping for (not necessary optimal) symmetric plan
$\pi$  with ergodic $X$-marginal implies ergodicity of  $\pi$.

\begin{proposition}
Let  $X=Y$ is be Polish space and $\Gamma$ be a group of Borel  one-to-one transformations acting on $X$.
Assume that $\pi$ and $\mu$ are $\Gamma$-invariant Borel probability measures on $X \times Y$ and $X$ respectively.
Assume, in addition, that $Pr_X \pi = \mu$, $\mu$ is  ergodic, and $\pi(\{x, T(x)\})=1$
for some Borel mapping $T$. Then $\pi$ is ergodic.
\end{proposition}
\begin{proof}
Assuming the contrary we represent $\pi$ as a convex combinations of two $\Gamma$-invariant measures
$$
\pi = \lambda \pi_1 + (1-\lambda) \pi_2,
$$
$\pi_1  \ne \pi_2$, $0 < \lambda < 1$. Clearly, this implies a similar decomposition for the projections $\mu = \lambda Pr_X \pi_1 + (1-\lambda) Pr_X \pi_2$.
If we show that $\mu_1$, $\mu_2$ are $\Gamma$-invariant and distinct, we will get a contradiction.
The $\Gamma$-invariance of both measures follows immediately from the $\Gamma$-invariance of $\pi_i$.
Let us show that $\mu_1 \ne \mu_2$.
Assume the contrary and take a Borel set $B \subset X \times Y$.
We get that $\pi_i(B)$ equals to $\mu_i(A)$, where $A =\mbox{\rm{Pr}}_X(B \cap \mbox{\rm{Graph}}(T)))$ 
(note that $A$ is universally measurable  as a projection of a Borel  set). Then it follows that $\pi_i$ coincide because $\mu_i$ do coincide.
\end{proof}

\section{Kantorovich duality}

In this section we start to study measures which are invariant under actions of
some group. 
The results of this section will not be used in this paper, but they are of independent interest.

Let $X$, $Y$ be Polish spaces, $\Gamma$ be a locally-compact amenable group with \emph{continuous} actions $L_\Gamma^X$, $L_\Gamma^Y$ on $X$, $Y$ respectively. The action $L_\Gamma$ on the product space $X\times Y$ is defined as follows:
$$
L_g(x,y)=(L_g(x), L_g(y)).
$$
where $L_g$ is an element of $L_\Gamma$ corresponding to $g\in \Gamma$.

Let us define the space $W_\Gamma\subset C_b(X\times Y)$ as the closure of linear span of the following set:
$$
\{f-f\circ L_g:\ f\in C_b(X\times Y),\ g\in \Gamma\}.
$$
It can be checked that the property
\begin{equation}
\label{invariance criterion}
\int \omega d\pi=0,\ \forall \omega\in W_\Gamma
\end{equation}
of a probability measure $\pi\in \PP(X\times Y)$ is equivalent to its invariance w.r.t. $L_\Gamma$.

Let $\mu\in \PP(X)$, $\nu\in \PP(Y)$ be invariant under the actions $L^X_\Gamma$, $L^Y_\Gamma$ respectively. Then a transport plan $\pi\in \Pi(\mu,\nu)$ is invariant iff the property (\ref{invariance criterion}) is satisfied. We denote the set of all invariant
 transport plans by $\Pi^\Gamma(\mu,\nu)$.

The following Theorem is a refinement of the duality result, which was proved in \cite{Zaev1} (Theorem 2.5). In there we considered only   $C_b(X\times Y)$ cost functions (we warn the reader that the classical duality statement from Section 2 is formulated in a slightly different but equivalent way:
in notations of this section $\Phi = \frac{x^2}{2} - \varphi, \Psi = \frac{y^2}{2} - \psi$).  
\begin{theorem}
	\label{Invariant Kantorovich duality}
	Let	$c \in C(X\times Y)$ be a nonnegative function such that there exist $f\in L^1(X,\mu)$, $g\in L^1(Y,\nu)$, and
	$$
	c(x,y)\leq f(x)+g(y),\ \forall (x,y)\in X\times Y.
	$$
	Then, in the setting described above,
	$$
	\inf_{\pi\in \Pi^\Gamma}{\int c d\pi}=\sup_{\Phi + \Psi + \omega \leq c}{\int_{X}{\Phi(x) d\mu} + \int_{Y}{\Psi(y) d\nu}},
	$$
	where $\Phi \in L^1(X)$, $\Psi\in L^1(Y)$, $\omega \in W_\Gamma$.
\end{theorem}
\begin{proof}
	The inequality
	$$
	\inf_{\pi\in \Pi^\Gamma}{\int c d\pi} \geq \sup_{\Phi + \Psi + \omega \leq c}{\int \Phi d\mu}+\int \Psi d\nu
	$$
	can be easily obtained:
	\begin{multline*}
	\inf_{\pi\in \Pi^\Gamma}{\int c d\pi} \geq \inf_{\pi\in \Pi^\Gamma}\left(\sup_{\Phi+\Psi + \omega \leq c}{\int (\Phi + \Psi + \omega) d\pi}\right)=\\
	=\inf_{\pi\in \Pi^\Gamma}\left(\sup_{\Phi + \Psi + \omega \leq c}{\int \Phi d\mu}+\int \Psi d\nu\right)=\sup_{\Phi + \Psi + \omega \leq c}{\int \Phi d\mu}+\int \Psi d\nu.
	\end{multline*}
	To obtain the opposite inequality we use the following statement from Theorem 2.5 of \cite{Zaev1}.
	$$
	\inf_{\pi\in \Pi^\Gamma}{\int c_b d\pi}=\sup_{\Phi + \Psi + \omega \leq c_b}{\int_{X}{\Phi(x) d\mu} + \int_{Y}{\Psi(y) d\nu}}
	$$	
	for $c_b\in C_b(X\times Y)$, $\Phi \in C_b(X)$, $\Psi\in C_b(Y)$, $\omega \in W_\Gamma$.
	Let $c_n(x,y):=\min\{c(x,y), n\}$ for each $n\in N$. The inequality
	$$
	\sup_{\Phi + \Psi + \omega \leq c_n}{\int_{X}{\Phi(x) d\mu} + \int_{Y}{\Psi(y) d\nu}}\leq \sup_{\Phi + \Psi + \omega \leq c}{\int_{X}{\Phi(x) d\mu} + \int_{Y}{\Psi(y) d\nu}}
	$$
	is obvious for any natural $n$. Thus it remains to prove that 
	$$
	\lim_{n\to \infty}\inf_{\pi\in \Pi^\Gamma}{\int c_n d\pi}=\inf_{\pi\in \Pi^\Gamma}{\int c d\pi}.
	$$
	Recall that the functional $\pi\rightarrow \int c_b d\pi$ is weakly continuous for every $c_b\in C_b(X\times Y)$. It follows from the characterization (\ref{invariance criterion}) of invariant measures, that $\Pi^\Gamma(\mu,\nu)$ is a closed
	 subset of $\Pi(\mu,\nu)$, which is known to be compact. Thus $\Pi^\Gamma(\mu,\nu)$ is compact in the topology of weak convergence. If $\pi_n$ is the solution for 
	$$
	\inf_{\pi\in \Pi^\Gamma}{\int c_n d\pi},
	$$
	the sequence $(\pi_n)$ has to have a subsequence converging to some element $\pi^*\in \Pi^\Gamma$. Since for any fixed $m\in \mathbb{N}$ the inequality: $\lim_{n\to \infty}{\int c_n d\pi^*}\geq {\int c_m d\pi^*}$ is satisfied, and,
	by monotone convergence theorem, $\lim_{m\to \infty}{\int c_m d\pi^*}=\int c d\pi^*\leq \int (f(x) + g(y)) d\pi^*<\infty$, we obtain
	$$
	\lim_{n\to \infty}{\int c_n d\pi_n}\geq \lim_{m\to \infty}{\int c_m d\pi^*}=\int c d\pi^*\geq \inf_{\pi\in \Pi^\Gamma}{\int c d\pi}.	
	$$
	This fact concludes the proof of the theorem.
\end{proof}

As one can see, the form of the duality theorem is similar to the well-known classic result, but the difference is substantial: dual functionals are related to each other in a more complicated way. Moreover, there is no existence result for the dual problem without any additional assumptions. 

It was shown in \cite{Zaev1} (Theorem 5.7) that in case of compact group $\Gamma$ and under the assumptions of Theorem \ref{Invariant Kantorovich duality},
	$$
	\inf_{\pi\in \Pi^\Gamma}{\int c d\pi}=\sup_{\Phi + \Psi \leq \bar{c}}{\int_{X}{\Phi(x) d\mu} + \int_{Y}{\Psi(y) d\nu}}.
	$$
where $\bar{c}:=\int_{\Gamma} (c\circ g) d\chi(g)$ and $\chi(g)$ is the probability Haar measure. It is clear that if cost function is $\Gamma$-invariant, the invariant dual problem coincides with the usual one.

Moameni (\cite{Moameni}) proved that for $\Gamma=\mathbb{Z}$ and an invariant cost function $c$, the corresponding invariant dual problem coincides with the usual one, and, moreover, both prime and dual Kantorovich problems have an invariant solution.

\section{Existence of invariant optimal mapping for stationary measures}

Recall that the measures on $\mathbb{R}^{\infty}$ which are invariant with respect to the shift
$$
\sigma(x_1, x_2, \ldots) =  (x_2, x_3, \ldots)
$$
are called stationary measures.
Unlike  exchangeable measures, the projections of stationary measures are in general not  invariant with respect to some reasonable family of  linear transformation.

As usual we assume that $\mathbb{R}^{\infty}$ is approximated by the sequence of finite-dimensional spaces  $\mathbb{R}^{n}$
in the following sense: we identify $\mathbb{R}^{n}$ with the subset
$$
P_n (\mathbb{R}^{\infty}) = \{ x = (x_1, x_2, \cdots, x_n, 0, 0, \cdots)\} \subset  \mathbb{R}^{\infty}.
 $$ 
 On every finite-dimensional space $\mathbb{R}^n$ we will apply the following operator of cyclical shift:
$$\sigma_n (x_1, x_2, \cdots, x_n) = (x_2, x_3, \cdots, x_{n}, x_1).
$$
Let us associate with every stationary measure $\mu$ the cyclical average of its projections:
$$
\hat{\mu}_n = \frac{1}{n} \sum_{i=1}^n (\mu \circ P^{-1}_n) \circ \sigma^{-(i-1)}_n.
$$
In addition, let us denote by
$\mathbb{R}_{m,n}$ the orthogonal complement of $\mathbb{R}^m \subset \mathbb{R}^n$:
$$\mathbb{R}^n = \mathbb{R}^m \times \mathbb{R}_{m,n}, \ m <n.$$

The marginal measures are always assumed to satify the following property: 

{\bf Assumption A.} The measures $\mu,\nu$ are  stationary  Borel probability measures  such that
their projections on every $\mathbb{R}^n$
$$
 \mu \circ Pr_n^{-1}, \ \nu \circ Pr_n^{-1}
$$
have Lebesgue densities and bounded second moments.

We consider symmetric Monge-Kantorovich problem
\begin{equation}
\int (x_1 - y_1)^2 \ d \pi \to \min
\label{GMK}
\end{equation}
where the infimum is taken among of all  stationary measures $\Pi^{\Gamma}(\mu,\nu)$ with marginals $\mu,\nu$.

\begin{remark}
Minimizing $\int (x_1 - y_1)^2 \ d \pi $ is equivalent to maximizing of $\int x_1 y_1 \ d \pi$, because $\int x^2_1 \ d \pi = \int x^2_1 \ d \mu,
\ \int y^2_1 \ d \pi = \int y^2_1 d \nu$ are fixed.
\end{remark}

\begin{theorem}
\label{stationarytransport}
Let $\mu$ be a stationary measure which satisfies the following assumptions:
\begin{itemize}
\item[1)]
$\mu$ is a weak limit of a sequence of $\sigma_n$-invariant measures $\mu_n$
on $\mathbb{R}^n$.
\item[2)] For every $m<n$ there exists a probability measure $\mu_{m,n}$ on $\mathbb{R}_{m,n}$ such that the relative entropy (the Kullback-Leibler distance) between $\mu_m \times \mu_{m,n}$ and $\mu_n$ is uniformly bounded in $n$:
$$
\int \log \Bigl( \frac{ d\mu_n}{d (\mu_m \times \mu_{m,n})}\Bigr) d \mu_n < C_m
$$
with $C_m$ satisfying
$$
\lim_m \frac{C_m}{m}=0;
$$
\item[3)] 
The cyclical average  $\hat{\mu}_n$ of the $n$-dimensional projection $\mu \circ P^{-1}_n$
has finite second moments and admits a density $\rho_n$  with respect to $\mu$
satisfying
$$
\sup_{n} \int \rho^{-\varepsilon}_n d \mu < \infty
$$
for some $\varepsilon >0$.
\end{itemize}
Then there exists a mapping $T$  with the properties
\begin{itemize}
\item
$T$ pushes forward $\mu$ onto  the standard Gaussian measure on $\mathbb{R}^{\infty}$:
$$
\nu = \gamma.
$$
\item
 $T$ a $\mu$-a.e. limit of finite dimensional mappings
$T_n : \mathbb{R}^n \mapsto \mathbb{R}^n$ such that every $T_n$ is a solution to an optimal transportation problem
on $\mathbb{R}^n$.
\end{itemize}
\end{theorem}
\begin{proof}
 We consider the sequence of $n$-dimensional optimal transportation mappings $T_n$ with cost function $\sum_{i=1}^n (x_i-y_i)^2$ pushing forward
$\mu_n$ onto $\gamma_n$. It follows from the $\sigma_n$-invariance of $\mu_n$ and $\gamma_n$ that
the mapping $T_n$ is cyclically invariant:
$$
\langle T_n \circ \sigma_n, e_i \rangle =
\langle T_{n}, e_{i-1} \rangle, \ \ \mu_n-{\rm a.e.}
$$

Fix  a couple of numbers $m,n$ with $n>m$.
Let $T_{m,n}$ be the optimal transportation mapping  for the cost function $\sum_{i=n+1}^{m} (x_i-y_i)^2$  pushing foward $\mu_{m,n}$ onto the standard Gaussian measure on $\mathbb{R}_{m,n}$.
We stress that $T_m$ and $T_{m,n}$ depend on different collections of coordinates.

We extend $T_m$ onto $\mathbb{R}^n$ in the following way:
$$
T_m(x) = T_m(P_m x) + T_{m,n}(P_{m,n} x).
$$
Clearly, $T_m$ pushes forward $\mu_m \times \mu_{m,n}$ onto the standard  Gaussian measure on $\mathbb{R}_n$.
Applying Proposition \ref{generalTalagrand} to the couple of mappings $T_m, T_n$, we get
\begin{equation}
\label{entr-gibbs}
\frac{1}{2} \int \| T_n- T_m\|^2 d \mu_n \le \int  \log\Bigl( \frac{d \mu_n }{d (\mu_m \times \mu_{m,n})}\Bigr) d \mu_{n}.
\end{equation}
This implies
\begin{equation}
\label{mn-l2}
\sum_{i=1}^{m} \int \langle T_n - T_m, e_i \rangle ^2 \ d \mu _n   \le \int \| T_n- T_m\|^2 d \mu_n \le  2 C_m
\end{equation}
for every $m, n$, $m<n$.

Let us note that for every $i$ one can extract a weakly convergent subsequence from a sequence of (signed) measures $\{ \langle T_n, e_i \rangle \cdot \mu_n \}$.
Indeed, for any compact set $K$
$$
\Bigl(\int_{K^c} |\langle T_n, e_i \rangle | d \mu_n\Bigr)^2 \le  \int  |\langle T_n, e_i \rangle |^2  d \mu_n  \cdot \mu_n(K^c) = \int x^2_i \ d \gamma \cdot \mu_n(K^c) .
$$
Using the tightness of $\{\mu_n\}$ we get that $\{|\langle T_n, e_i \rangle | \cdot \mu_n\}$ is a tight sequence. In addition, note that for every continuous $f$
$$
\lim_n \Bigl( \int f |\langle T_n, e_i \rangle | d \mu_n\Bigr)^2 \le \int x^2_i \ d \gamma \cdot \int f^2 d \mu.
$$
This implies that any limiting point of $\{ \langle T_n, e_i \rangle \cdot \mu_n \}$ is absolutely continuous with respect to $\mu$.
Applying the diagonal method and passing to a subsequence one can assume that  convergence takes place for all $i$ simultaneously.
Consequently, there exists a subsequence $\{n_k\}$ and a measurable mapping $T$ with values in $\mathbb{R}^{\infty}$ such that
$$
\langle T_{n_k}, e_i \rangle \cdot \mu_{n_k} \to
\langle T, e_i \rangle \cdot \mu
$$
weakly in the sense of measures for every $i$.
It is easy to check that the standard property of $L^2$-weak convergence holds also in this case:
\begin{equation}
\label{mu-n-fat}
 \int  \langle T, e_i \rangle^2   d \mu
\le
\underline{\lim}_k \int  \langle T_{n_k}, e_i \rangle^2  d \mu_n
= \int x^2_i \ d \gamma =1.
\end{equation}

Finally, we  pass to the limit in (\ref{mn-l2}) and get
\begin{equation}
\label{m-l2}
\sum_{i=1}^{m} \int \langle T - T_m, e_i \rangle ^2 \ d \mu     \le  2C_m.
\end{equation}
The claim follows from (\ref{mu-n-fat}) and the fact that  $\lim_n \int \varphi \ d \mu_n = \int \varphi \ d \mu$ for every $\varphi \in L^2(\mu)$.
Indeed, if $\varphi$ is bounded and continuous, this follows from the weak convergence $\langle T_n, e_i \rangle \cdot \mu_n \to \langle T, e_i \rangle \cdot  \mu$. For arbitrary $\varphi \in L^2(\mu)$ we find 
 continuous bounded  cylindrical function  $\tilde{\varphi}$ such that $\| \varphi - \tilde{\varphi}\|_{L^2(\mu)}< \varepsilon$. One has
$ \lim_n \int \varphi \ d \mu_n = \lim_n\int (\varphi - \tilde{\varphi}) \ d \mu_n + \int \tilde{\varphi} \ d \mu$. The claim follows from the estimate
$$
\Bigl( \int |\varphi - \tilde{\varphi}| \ d \mu_n \Bigr)^2 \le \int (\varphi - \tilde{\varphi})^2 \ d \mu \cdot  \int \rho^2_n \ d \mu \le (\sup_n \int \rho^2_n \ d \mu )\varepsilon^2.
$$

Note that $T$ commutes with the shift $\sigma$: $\langle T \circ \sigma, e_i \rangle = \langle T, e_{i-1} \rangle$.
Indeed, for every  bounded cylindrical $\varphi$ one has
$$
\int \varphi \langle T_n, e_{i-1} \rangle d \mu_n = \int \varphi \langle T_n(\sigma_n), e_i \rangle d \mu_n = \int \varphi (\sigma_n^{-1}) \langle T_n, e_i \rangle d \mu_n
= \int \varphi (\sigma^{-1}) \langle T_n, e_i \rangle d \mu_n.
$$
Here we use  that $\varphi (\sigma_n^{-1}) = \varphi (\sigma^{-1})$ for sufficiently large values of $n$ and the cyclical invariance of $T_n$. Passing to the limit in the $n_k$-subsequence one gets
$$
\int \varphi \langle T, e_{i-1} \rangle d \mu =  \int \varphi (\sigma^{-1}) \langle T, e_i \rangle d \mu =  \int \varphi \langle T \circ \sigma, e_i \rangle d \mu.
$$
Hence $T \circ \sigma = \sigma \circ T$.

Hence by assumptions of the theorem and (\ref{m-l2}) we get
$$
\limsup_m \frac{1}{m} {\sum_{i=1}^{m} \int \langle T - T_m, e_i \rangle ^2 \ d \mu }=0.
 $$ 

To prove that $T$ pushes forward $\mu$ into $\gamma$ it is sufficient to show that that $\langle {T}_m, e_i \rangle \to \langle {T}, e_i \rangle$ in measure (see Lemma \ref{conv-lemm}). To this end
 let us approximate $T_1$ by a bounded function $\xi_1(x_1, \ldots,x_k)$ depending 
on finite number of coordinates in $L^2(\mu)$: $\int \| T_1 - \xi_1\|^2 d \mu < \varepsilon$, where 
$\varepsilon$ is chosen sufficiently small. Set: $\xi_i = \xi \circ \sigma^{i-1}$.
Clearly, we get by the shift invariance
$$
 \frac{1}{m} \int \sum_{i=1}^m (T_i-\xi_i)^2 d \mu = \int (T_1 - \xi_1)^2  d\mu < \varepsilon.
$$
Hence
$$
\limsup_m \frac{1}{m} \int \|  {T}_m - \xi \|^2 d \mu \le \varepsilon, \  \xi = (\xi_1,\xi_2, \ldots).
$$
Let make the change of variables under the cyclical shift $\sigma_n$. One has
$$
\langle {T}_m , e_i \rangle \circ \sigma^{-(i-1)}_m = {T}_1
$$
for all $ 1 \le i \le m$
and
$$
 \xi_i \circ \sigma^{-(i-1)}_m = \xi_1
$$
as soon as $i-1+k \le m$. Hence for the latter values of $i$ one has
$$
\int \langle \xi -  {T}, e_i \rangle^2 d \mu = \int \langle \xi - {T}, e_1 \rangle^2 \ d \mu \circ \sigma_n^{i}.
$$
The number of indices which do not satify this property is limited by $k$. Clearly, it doses not affect the limit of averages.
Finally we obtain
$$
\varepsilon \ge \limsup_m \frac{1}{m} \int \sum_{i=1}^n \langle \xi - {T}_m, e_i \rangle^2 d \mu = \limsup_m \int  \langle 
\xi - {T}_m, e_1 \rangle^2  d \hat{{\mu}}_m.
$$
Recall that $\int (T_1 - \xi_1)^2 d \mu \le \varepsilon$.
Finally
\begin{align*}
\limsup_m  \int  \langle 
T- {T}_m, e_1 \rangle^2  d \hat{{\mu}}_m & \le 2 \limsup_m
\int  \langle 
\xi - {T}_m, e_1 \rangle^2  d\hat{{\mu}}_m \\&  + 2 \limsup_m \int (T_1 - \xi_1)^2 d \hat{{\mu}}_m
\le 4 \varepsilon.
\end{align*}
Since $\varepsilon>0$ is arbitrary, one gets $\int \langle T - {T}_m, e_1 \rangle^2 d \hat{\mu}_m \to 0$. 
By the H{{\" o}}lder inequality 
$$
\int \langle T - {T}_m, e_1 \rangle^{\frac{2}{p}} d \mu \le \Bigl( \int  \langle T - {T}_m, e_1 \rangle^{2} d  \hat{{\mu}}_m \Bigr)^{\frac{1}{p}} \Bigl( \int \rho_m^{-\frac{1}{p-1}} d \mu \Bigr)^{\frac{1}{q}}.
$$
Take $p = 1 + \frac{1}{\varepsilon}$ we get by the assumption of the theorem that the latter tends to zero. The proof is complete.
\end{proof}

\begin{remark}
In Theorem \ref{stationarytransport} the Gaussian measure $\gamma$ can be replaced by any countable power of an uniformly log-concave one-dimensional measure. 
\end{remark}

In the following proposition we prove that the transportation mapping $T$ is indeed
optimal under additional assumptions.

\begin{proposition}
\label{30.10.15}
Let the assumptions of Theorem \ref{stationarytransport}  hold.
Assume in addition that
$$
\lim_{n \to \infty} \frac{1}{n}  W^2_2(\hat{\mu}_n,\mu_n)=0. 
$$
Then there exists a solution $\pi$ of  problem (\ref{GMK}) in the class of stationary measures
such that $\pi \{(x, T(x)), x \in \mathbb{R}^{\infty}\}=1$.
\end{proposition}
\begin{proof}
We show that the measure $\pi = \mu \circ (x,T(x))^{-1}$, which is the weak limit of measures $\pi_n$
is optimal. Recall that 
$\pi_n$ gives minimum to 
$
m \to \int \sum_{i=1}^n (x_i - y_i)^2 d m
$
and has marginals $\mu_n,\gamma_n$, hence measure $\pi$ has marginals $\mu,\gamma$.
Indeed,
$$
\int (x_1 - y_1)^2 d {\pi} = \lim_n \int (x_1 - y_1)^2 d {\pi}_n = \lim_n \frac{1}{n} \int \sum_{i=1}^n (x_i -  y_i)^2 d \pi_n.
$$
If ${\pi}$ is not optimal, when there exists a stationary measure ${\pi}_0$ with projections $\mu,\nu$ such that
$$
\int (x_1 - y_1)^2 d \pi_0  +  \varepsilon<   \frac{1}{n} \int \sum_{i=1}^N (x_i - y_i)^2 d \pi_n
 $$ 
 for some $\varepsilon>0$ and all sufficiently big values of $n$. Taking into account stationarity of $\pi_0$ we get $\int  x_i y_i d \pi_0 = \int  x_j y_j \pi_0$ for every $i,j$, thus
 $$
 \int \sum_{i=1}^n (x_i - y_i)^2 d\hat{\pi}_0 + n \varepsilon  = \int \sum_{i=1}^n (x_i - y_i)^2 d \pi_0 + n \varepsilon < \int \sum_{i=1}^n (x_i - y_i)^2 d \pi_n,
 $$
where  $\hat{\pi}_0 = \frac{1}{n} \sum_{i=1}^{n} (\pi_0 \circ Pr_{n}^{-1}) \circ \sigma^{-(i-1)}_{n}$ . The latter inequality implies
$$
W^2_2(\hat{\mu}_n,\gamma_n) + n \varepsilon \le  W^2_2(\mu_n,\gamma_n) .$$
By the triangle inequality
\begin{align*}
W^2_2(\hat{\mu}_n,\gamma_n) + n \varepsilon & \le (W_2(\mu_n, \tilde{\mu}_n) + W_2(\hat{\mu}_n, \gamma_n) )^2 
\\& 
\le W^2_2(\mu_n, \hat{\mu}_n) + 2 W_2(\hat{\mu}_n, \gamma_n)W_2(\mu_n, \hat{\mu}_n)  +  W^2 _2(\hat{\mu}_n, \gamma_n) .
\end{align*}
Hence
\begin{equation}
\label{30.10}
\varepsilon \le \frac{1}{n} (2 W_2(\hat{\mu}_n, \gamma_n)W_2(\mu_n, \hat{\mu}_n)  +  W^2 _2(\hat{\mu}_n, \mu_n) ).
\end{equation}
The quantity $W^2_2(\hat{\mu}_n, \gamma_n)$ can be trivially estimated by 
$2\sum_{i=1}^n (\int x_i^2 d \hat{\mu}_n + \int y_i^2 d \gamma_n ) \le Cn$. Then the using the assumption of the theorem we get that the right-hand side of (\ref{30.10})
tends to zero, which contradicts to positivity of $\varepsilon$.
\end{proof}

 We finish this section with a concrete application of Theorem
\ref{stationarytransport}. We study a transportation of a Gibbs measure $\mu$ which can be formally written in the form
$$
\mu =  e^{-H(x)} dx,
$$
where the potential $H$ admits the following heuristic representation:
$$
H(x) = \sum_{i=1}^{\infty} V(x_i) + \sum_{i= 1}^{\infty} W(x_i,x_{i+1}).
$$
Here $V$ and $W$ are smooth functions and $W(x,y)$ is symmetric: $W(x,y)=W(y,x)$.
The existence of such measures was proved in \cite{AKRT}.

Let us specify the assumptions about $V$ and $W$.
These are a particular case of assumptions A1-A3 from \cite{AKRT}.

\begin{itemize}
\item[1)]
$$
W(x,y) = W(y,x);
$$
\item[2)]
There exist numbers $J>0$, $L \ge 1$, $N \ge 2$, $\sigma>0$, and $A,B,C>0$ such that
$$
| W(x,y) | \le J(1+ |x|+|y|)^{N-1}, \
| \partial_{x}  W(x,y) | \le J(1+ |x|+|y|)^{N-1}
$$
\item[3)]
$$
|V(x)| \le C(1+|x|)^L, \ \
|V'(x)| \le C(1+|x|)^{L-1};
$$
\item[4)]
(coercivity assumption)
$$
 V'(x) \cdot x \ge A|x|^{N + \sigma} -B.
$$
\end{itemize}

Let us define the following probability measure on $E_n$:
$$
\mu_n = \frac{1}{Z_n} \exp\Bigl( -\sum_{i=1}^{n}\bigl( V(x_i) +  W(x_i, x_{i+1}) \bigr) \Bigr),
$$
with the convention $x_{n+1} :=x_{1}$. Here $Z_n$ is the normalizing constant.

\begin{proposition}
\label{stationarymap}
The sequence $\mu_n$ admits a weakly convergent subsequence
$\mu_{n_k} \to \mu$ satisfying the assumptions of Theorem \ref{stationarytransport}.
\end{proposition}
\begin{proof}
It was proved in  Theorem 3.1 of \cite{AKRT}
that any sequence of probability measures $$\tilde{\mu}_n = c_n e^{-H_n} dx_{-n} \cdots dx_n,$$ where
$H_n$ is obtained from $H$ by fixing a boundary condition $\tilde{x}$
$$
H_n = \sum_{i=1}^{n} V(x_i) + \sum_{i=1}^{n-1} W(x_i, x_{i+1}) +    W(x_{n}, \tilde{x}_{1}),
$$
has a weakly convergent subsequence $\tilde{\mu}_{n_k} \to \tilde{\mu}$.
In addition (see \cite{AKRT}), $\mu$ satisfies  the following a priori estimate: for every $\lambda>0$
$$
\sup_{k \in \mathbb{N}} \int \exp(\lambda |x_k|^N) \ d \tilde{\mu} < \infty.
$$ The same estimate holds for $\tilde{\mu}_n$ uniformly in $n$.

Following the reasoning from \cite{AKRT} it is easy to show that the sequence $\{\mu_n\}$ is tight and satisfies the same a priori estimate.
Thus, we can pass to a subsequence $\{\mu_{n'}\}$ which weakly converges to a measure $\mu$.
For the sake of simplicity this subsequence will be denoted by $\{\mu_n\}$ again. The limiting measure $\mu$  satisfies
\begin{equation}
\sup_{k \in \mathbb{N}} \int \exp(\lambda |x_k|^N) \ d {\mu} < \infty,
\end{equation}
moreover,
\begin{equation}
\label{exp-ap-est}
\sup_n \sup_{k \in \mathbb{N}} \int \exp(\lambda |x_k|^N) \ d {\mu_n} < \infty.
\end{equation}

Let us estimate the relative entropy. We note that $\mu_n$ and $\mu_m$ ($n>m$) are related in the following way:
$$
\frac{e^{Z} \mu_n}{\int e^{Z} d \mu_n}  = \mu_m \times \nu_{m,n},
$$
where $Z = -W(x_m, x_{1}) + W(x_m,x_{m+1})+ W(x_n, x_1)$, and $\nu_{m,n}$ is a probability measure on $E_{m,n}$. Set: $\mu_{m,n} = \nu_{m,n}$.
Then
$$
\int \log \Bigl( \frac{ d\mu_n}{d (\mu_m \times \mu_{m,n})}\Bigr) d \mu_n = \int (Z - \log \int e^Z \ d \mu_n) \ d \mu_n.
$$
The desired bound follows immediately from (\ref{exp-ap-est}) and the assumptions about $W$.

In order to prove assumption 3) we note that
$$
\frac{ \Bigl[ e^{W(x_n,x_{n+1}) + W(x_{1}, x_n)} \cdot \mu \Bigr] \circ P^{-1}_n}{\int e^{W(x_n,x_{1}) + W(x_{1}, x_{n})} d \mu} =
\frac{ e^{W(x_1,x_{n})} \cdot \mu_n}{\int e^{W(x_1,x_{n})} \ d \mu_n}.
$$
The normalizing constants can be easily estimated with the help of a priori bounds for $\mu$ and $\mu_n$. Applying assumptions on $W$ one can easily get that
$$
A e^{-B(|x_n|^{N-1} + |x_{q}|^{N-1})} \le  \frac{d \mu_n}{d \mu \circ P^{-1}_n} \le A e^{B(|x_n|^{N-1} + |x_{1}|^{N-1})}
$$
where $A, B >0$  do not depend on $n$. Hence, Assumption 3) follows immediately from (\ref{exp-ap-est}), the Jensen inequality and convexity if the function $x^{-\varepsilon}$.
\end{proof}

\begin{remark}
Finally, let us briefly discuss when the transportation mapping obtained in Proposition \ref{stationarymap} by  Theorem \ref{stationarytransport}
solves the corresponding optimal transportation problem.
To this end we apply Proposition \ref{30.10.15}.

Following the  estimates obtained in Proposition \ref{stationarymap} and applying Jensen inequality
one can easily show that the sequence of the entropies
$$ 
\int \log \Bigl( \frac{d \hat{\mu}_n}{d \mu_n} \Bigr) \ d {\hat{\mu}}_n 
$$
is bounded. Then the assumption of Proposition \ref{stationarymap} holds, for instance, if  every $\mu_n$ satisfies the Talagrand  inequality
$$
W^2_2(\mu_n, \rho \cdot \mu_n) \le C \int \rho \log \rho d \mu_n
$$
with constant which does not depends on $n$.
We don't investigate  here sufficient condition for measures $\mu_n$ to satisfy this inequality, we just mention that this clearly holds in many natural situations
(e.g. under assumption of uniform log-concavity or finiteness of the log-Sobolev constant).

In addition, we emphasize, that in many applications the measures do indeed satisfy the Talagrand inequality, but Proposition \ref{30.10.15} should actually work under much milder assumptions.
\end{remark}


\begin{thebibliography}{10}




\bibitem{AGS}
Ambrosio~L., Gigli~N., Savar{\'e}~G.,
 Gradient flows in metric
spaces and in the Wasserstein spaces of probability measures, Birkh{\"a}user,
  2008.

\bibitem{AKRT}
Albeverio~S., Kondratiev~Yu.G., R{\"o}ckner~M., Tsikalenko~T.V.,
A priori estimates for symmetrizing measures and their applications to Gibbs states, Journ. of Func. Anal., 171, 366--400, 2000.

\bibitem{B}
Beiglb\"ock~M., Cyclical monotonicity and the ergodic theorem,. Ergod. Theory and Dynam. Syst., 35(3),  710--713, 2015.

\bibitem{Bo}
Bogachev V.I.,
Measure theory. V.~1,2. Springer, Berlin -- New York, 2007.


\bibitem{BoKo2005}
Bogachev~V.I., Kolesnikov~A.V.,
On the  Monge--Amp\`ere equation in infinite dimensions,
Infin. Dimen. Anal.
Quantum Probab. Related Topics,  8(4), 547--572, 2005.


\bibitem{BoKo2011}
Bogachev~V.I., Kolesnikov~A.V.,
Sobolev regularity for the Monge--Ampere equation in the Wiener space. Kyoto Jour. Math., 53(4), 713--738, 2013.


\bibitem{BoKo2012}
Bogachev~V.I., Kolesnikov~A.V., The Monge--Kantorovich problem: achievements, connections, and perspectives,
Russian Mathematical Surveys, 67(5), 785--890, 2012.




\bibitem{Caval}
Cavalletti~F., The Monge problem in Wiener Space. Calcul. of Var.and PDE's, 45(1-2), 101--124, 2012.



 \bibitem{CLO} 
Contreras~G., Lopes~A.O., Oliveira~E.R., 
Ergodic Transport Theory, periodic maximizing probabilities and the twist condition. In:
Modeling, Dynamics, Optimization and Bioeconomics I, vol.~ 73 of the series Springer Proceedings in Mathematics and Statistics, 183--219, 2014.

\bibitem{FangNolot}
Fang~S., Nolot~V., Sobolev estimates for optimal transport maps
on Gaussian spaces, Jour. Func. Anal., 266(8),  5045--5084, 2014.

\bibitem{FangShao}
Fang~S., Shao~J.
Optimal transport maps for Monge--Kantorovich problem on loop groups,
Journ. of Func. Anal., 248(1),  225--257, 2007.


\bibitem{FU1}
Feyel~D., {\"U}st{\"u}nel~A.S.,
Monge--Kantorovich measure transportation
and Monge--Amp{\`e}re equation on Wiener space, Prob. Theory and Related
Fields, 128,  347--385, 2004.


\bibitem{GhMau}
Ghoussoub~N., Moameni~A., 
Symmetric Monge-Kantorovich problems and polar decompositions of vector fields,
    Geom. Func. Anal., 24(4),  1129--1166, 2014.

\bibitem{Kallenberg}
Kallenberg~O., Probabilistic symmetries and invariance principles, Springer-Verlag New York, 2005.


\bibitem{Kol04}
Kolesnikov A.V., Convexity inequalities and optimal transport of infinite-dimensional measures,
 J. Math. Pures Appl. (9), 83(11),   1373--1404, 2004.


\bibitem{Kol2010}
Kolesnikov~A.V.,   On Sobolev regularity of mass transport and transportation inequalities, Theory of Probability and its Applications, to appear.
Translated from  Teor. Veroyatnost. i Primenen., 57(2),  296--321, 2012.


\bibitem{Kol-contr}
Kolesnikov~A.V., Mass transportation and contractions, MIPT Proc.  2(4),  90--99, 2010.


\bibitem{KR}
Kolesnikov~A.V., R{\"o}ckner~M.,  On transport equation in infinite dimensions. Jour. Func. Anal.
266(7),    4490--4537, 2014.


\bibitem{LM}
Lopes~A.O., Mengue~J.K., Duality Theorems in Ergodic Transport,
Journ~Stat.~Phys., 149(5),   921--942, 2012. 


\bibitem{LOT}
Lopes~A.O., Oliveira~E.O.,Thieullen~P.,
The Dual Potential, the involution kernel and
Transport in Ergodic Optimization, Dynamics, Games and Science
Vol. 1 of the series CIM Series in Mathematical Science,  357--398.

 \bibitem{Moameni}
Moameni A., Invariance properties of the Monge-Kantorovich mass transport problem,
arXiv: 1311.7051.

 \bibitem{RS} 
R{\"u}schendorf~L., Sei~T., On optimal stationary couplings between stationary processes
Electr.~Journ.~Probab., 17, article~17, 2012.   

\bibitem{Vershik}
Vershik~A.M., 
The problem of describing central measures on the path spaces of graded graphs., Func.l Analysis and Its Appl., 48(4), 256--271, 2014.
  

\bibitem{Vill}
Villani~C.,
 Topics in optimal transportation,
Amer. Math. Soc. Providence, Rhode Island, 2003.

  

\bibitem{Zaev1} Zaev~D.A., On the Monge-Kantorovich problem with additional linear constraints, Mat. Zametki, 98(5), 664--683, 2015.


\bibitem{Zaev2}
Zaev~D.A., 
On ergodic decompositions related to the Kantorovich problem,
Zapiski POMI, 437, 100--130, 2015.

\end{thebibliography}
\end{document}